\documentclass[12pt,oneside,showkeys]{article} 
\usepackage{amsmath, amsthm, amssymb, calrsfs, wasysym, verbatim, bbm, color, graphics, geometry,braket,enumerate}
\usepackage[style=alphabetic,doi=false,isbn=false,url=false,eprint=false]{biblatex} 
\makeatletter

\newcommand{\Rmnum}[1]{\expandafter\@slowromancap\romannumeral #1@}
\makeatother
\numberwithin{equation}{section}

\newtheorem{thm}{Theorem}
\newtheorem{defn}{Definition}

\newtheorem{lem}{Lemma}

\newtheorem{eg}{Example}
\title{Quantitative Obata's theorem in discrete setting}

\author{Shiping Liu\thanks{School of Mathematical Sciences, University of Science and Technology of China, Hefei 230026, China.
 Email address: spliu@ustc.edu.cn} \and Chiyu Zhou\thanks{School of Mathematical Sciences, University of Science and Technology of China, Hefei 230026, China. Email address: dovong@mail.ustc.edu.cn }}
\date{\today}

\begin{document}

\maketitle

\vspace{.25in}

\begin{abstract}
Under mild assumptions, we show that a connected weighted graph $G$ with lower Ricci curvature bound $K>0$ in the sense of Bakry-\'Emery and the $d$-th non-zero Laplacian eigenvalue $\lambda_d$ close to $K$, with $d$ being the maximal combinatorial vertex degree of $G$, has an underlying combinatorial structure of the $d$-dimensional hypercube graph. 
Moreover, such a graph $G$ is close in terms of Frobenius distance to a properly weighted hypercube graph.
Furthermore, we establish their closeness in terms of eigenfunctions. Our results can be viewed as discrete analogies of the almost rigidity theorem and quantitative Obata's theorem on Rimennian manifolds.

\end{abstract}
\section{Introduction}

There are quite some deep analogies between the $d$-dimensional round sphere $S_d$ and hypercube graph $H_d$. For example, in both cases, the concentration of measure converges to the Gaussian measure when the dimension $d$ tends to infinity \cite[Section 3$\frac{1}{2}$.21]{gromov}; Cheng's maximal diameter theorem for spheres also holds for hypercube graphs \cite{rigidityproperties}. However, Obata type eigenvalue rigidity theorem for hypercube graphs \cite{rigidityproperties} has certain essential differences from that for spheres. 

In Riemannian geometry,  the Lichnerowicz estimate relates Laplacian eigenvalues and Ricci curvature. In fact, for a $d$-dimensional complete Riemannian manifold $M$ with $\mathrm{Ric}\geq K>0$, its first non-zero eigenvalue $\lambda_1(M)\geq \frac{d}{d-1}K$. Obata's rigidity theorem tells that the equality holds if and only if $M$ is a $d$-dimensional round sphere of suitable radius. Recall that the unit sphere $S_d$ has constant Ricci curvature $d-1$ and its first non-zero eigenvalue $\lambda_1(S_d)$ is equal to $d$ with multiplicity $d+1$, i.e., \[d=\lambda_1(S_d)=\cdots=\lambda_{d+1}(S_d)<\lambda_{d+2}(S_d).\] Improving previous works of Petersen \cite{petersen}, Aubry \cite{aubry}
showed that a $d$-dimensional complete Riemannian manifold with $\mathrm{Ric}\geq d-1$ and $\lambda_d(M)$ close to $d$ is both Gromov-Hausdorff close and diffeomorphic to the unit sphere $S_d$. This is known as \emph{almost rigidity} theorem. For recent progress on almost rigidity results on K\"ahler manifolds, see \cite{kahler}. 

Furthermore, we recall that a function $u:S_d\to \mathbb{R}$ is an eigenfunction to $\lambda_1(S_d)$ with $\Vert u\Vert_{L^2(S_d)}=\mathrm{vol}(S_d)$ if and only if there exists a point $p\in S_d$ such that $u=\sqrt{d+1}\cos \mathrm{dist}_p$, where $\mathrm{dist}_p$ is the distance function to the point $p$ on $S_d$. Cavalletti, Mondino and Semola \cite{quantitativeobata} proved that if $u$ is a first eigenfunction of a $d$-dimensional complete Riemannian manifold $M$ with $\mathrm{Ric}\geq d-1$ and satisfies $\Vert u \Vert_{L^2(M)}=\mathrm{vol}(M)$, then there exists a point $p\in M$ such that
\[\left\Vert u-\sqrt{d+1}\cos \mathrm{dist}_p\right\Vert_{L^2(M)}\leq C(d)(\lambda_1(M)-d)^{\frac{1}{8d+4}}.\]
This improves previous works in \cite{petersen, pincementbertrand}. Such kind of results can be considered as a quantitative Obata's theorem. 

In discrete setting, an analogue of the Ricci curvature condition $\mathrm{Ric}\geq K$ is given by the Bakry-\'Emery curvature-dimension inequalities $CD(K,\infty)$ \cite{bakry,elworthy,riccicurlinyau,michael}.
Recall that the $d$-dimensional hypercube graph $H_d$ is a graph with vertex set $V=\{0,1\}^d$ and $\{u,v\}\in E$ is an edge if and only if they differ in exactly one coordinate. 
The graph $H_d$ satisfies $CD(2,\infty)$ and its first non-zero Laplacian eigenvalue is equal to $2$ with multiplicity $d$, i.e.
\[2=\lambda_1(H_d)=\cdots=\lambda_{d}(H_d)<\lambda_{d+1}(H_d).\]
However, a graph $G$ satisfying $CD(2,\infty)$ and $\lambda_1(G)=2$  is not necessarily a hypercube. That is, a directly analogue of Obata's theorem in discrete setting cannot be true. By requiring in addition that the first non-zero eigenvalue has a high multiplicity, an Obata type theorem has been established for graphs in \cite{rigidityproperties}, which will be recalled in details as Theorem \ref{rigiditytheorem} below. 

Let us still use $\mathrm{dist}_p$
to denote the combinatorial distance to a given vertex $p$ of $H_d$. Then, for any vertex $p$, the function \[\mathrm{dist}_p(\cdot)-\frac{d}{2}\] is an eigenfunction of $H_d$ to the eigenvalue $\lambda_1(H_d)$. However, for $d>1$, there exist eigenfunctions of $H_d$ corresponding to $\lambda_1(H_d)$ which can not be expressed as a function composed with the distance function, see Example \ref{examplecube}. This provides another difference between hypercube graphs and round spheres. 

Keeping those analogies and differences between hypercube graphs and round spheres in mind, we explore in this paper analogous results of the almost rigidity theorem and quantitative Obata's theorem in discrete setting.

\subsection{Statements of results}
We work in the setting of weighted graphs. A weighted graph $G=(V,w,m)$ is a triple, where $V$ is the set of vertices, $w:V\times V\to [0,\infty)$ is symmetric and vanishes on the diagonal, and $m: V\to (0,\infty)$ is a vertex measure. There is an underlying \emph{combinatorial structure} $\tilde{G}:=(V,E)$ of $G=(V,w,m)$, where the set $E$ of edges is defined as $E:=\{\{x,y\}\in V\times V: w(x,y)>0\}$. We say $G=(V,w,m)$ is connected if $\tilde{G}$ is connected. For any $x\in V$, we define its combinatorial vertex degree as $\mathrm{deg}(x):=\#\{y:\{x,y\}\in E\}$, and its weighted vertex degree as $\mathrm{Deg}(x):=\frac{1}{m(x)}\sum_{y\in V}w(x,y)$. Let us denote $\mathrm{Deg}_{\max}:=\sup_{x\in V}\mathrm{Deg}(x)$ and $\mathrm{deg}_{\max}:=\sup_{x\in V}\mathrm{deg}(x)$. We say the graph $G$ is locally finite if $\mathrm{deg}(x)<\infty$ holds for any $x\in V$.

For an edge $\{x,y\}\in E$, there are two corresponding oriented edges $(x,y)$ and $(y,x)$. We define the edge degree $q(x,y)$ of an oriented edge $(x,y)$ as $q(x,y):=w(x,y)/m(x)$. We say a weighted graph $G=(V,w,m)$ has constant edge degree $c$, if $q(x,y)=q(y,x)=c$ holds for any edge $\{x,y\}\in E$. 

It is known that a connected locally finite weighted graph satisfying Bakry-\'Emery curvature condition $CD(K,\infty)$ with $K>0$ and $\mathrm{Deg}_{\max}<\infty$ is finite \cite{diameterbounds}.  

Below is a discrete version of the Lichnerowicz eigenvalue estimate (see e.g. \cite{curvatureaspectsofgraphs,curvatureandhigherorder}). 
\begin{thm}[A discrete Lichnerowicz estimate]
    Let $G = (V,w,m)$ be a finite connected weighted graph satisfying $CD(K,\infty)$ for some $K>0$. Let~$0=\lambda_0<\lambda_1\leq\lambda_2\leq\cdots\leq \lambda_{|V|-1}$ be the eigenvalues of the graph Laplacian $-\Delta$. Then we have \(\lambda_1\geq K\).
\end{thm}

Next, we recall the Obata type rigidity theorem proved in \cite[Theorem 2.12]{rigidityproperties}. 
\begin{thm}[A discrete Obata's rigidity theorem]\label{rigiditytheorem}
    Let $G = (V, w, m)$ be a finite connected weighted graph. Let $0 = \lambda_0 <\lambda_1 \leq \lambda_2 \leq \cdots\leq \lambda_{|V|-1}$ be the eigenvalues of the graph Laplacian $-\Delta$. Then, $G$ satisfies $CD(K, \infty)$ and $\lambda_{\mathrm{deg}_{\max}} = K$ for some $K>0$ if and only if $G$ is the $\mathrm{deg}_{\max}$-dimensional weighted hypercube graph with constant edge degree $K/2$.
\end{thm}
For convenience, we denote by $H_d(K/2)$ the $d$-dimensional weighted hypercube graph with constant edge degree $K/2$.
Notice that we can not replace the condition $\lambda_{\mathrm{deg}_{\max}} = K$ in Theorem \ref{rigiditytheorem}  by $\lambda_{\ell} = K$ for some positive integer $\ell<\mathrm{deg}_{\max}$. Counterexamples can be found in \cite[Example 3.2]{rigidityproperties}. This demonstrate some essential difference between the discrete and continuous Obata's rigidity theorem. 

In this paper, we aim at establishing quantitative versions of Theorem \ref{rigiditytheorem}. In order to state our results, we introduce the following class of graphs.

\begin{defn}\label{definition}
Let $D>0, \delta>1$ be two real numbers and $d$ be a positive integer. We define the set $\mathcal{G}(D,d,\delta)$ to be the collection of connected weighted graphs $G=(V,w,m)$ with $\mathrm{deg}_{\max}=d$ satisfying the following conditions:
\begin{itemize}
    \item [(i)] $\mathrm{Deg}_{\max}\leq D$;
  
    \item [(ii)] $\delta^{-1}\leq m(x)\leq \delta$, for any $x\in V$.
\end{itemize}    
\end{defn}

A natural metric to measure the similarity between two weighted graphs is the Frobenius distance $\mathrm{dist}_F$ (see Definition \ref{Frodist} below).  Our first result can be considered as an analogy to the almost rigidity result of Petersen \cite{petersen} and Aubry \cite{aubry}.

\begin{thm}\label{MainTheorem}
For any $D>0, \delta>1, d\in \mathbb{Z}_+$ and $K>0$, there exists a positive constant $\epsilon_{0}(D,K,d,\delta)$, such that any graph $G=(V,w,m)\in \mathcal{G}(D,d,\delta)$ satisfying $CD(K,\infty)$ and $\lambda_{d}\leq K+\epsilon_0$ has a combinatorial structure $\tilde{G}$ coinciding with the $d$-dimensional hypercube $H_d$.  Furthermore, the \textit{Frobenius distance} between $G$ and $H:=H_d(K/2)$ satisfies
   $$\operatorname{dist}_F(G,H)\leq C(\lambda_d-K)^{\frac{1}{2}},$$
  where $C=C(d,K)$ is a constant depending only on $d$ and $K$.
   \end{thm}

Next, we prove an analogy of the quantitative Obata's theorem \cite{petersen,pincementbertrand,quantitativeobata} in terms of the closeness of eigenfunctions. 

\begin{thm}\label{MainTheoremDistance}
  Let $G=(V,w,m)\in \mathcal{G}(D,d,\delta)$ be a graph satisfying $CD(K,\infty)$ for some $K>0$. Let $\Lambda(d)$ be the space spanned by a family of  orthonormal eigenfunctions $\varphi_1,\ldots, \varphi_d$ corresponding to $0<\lambda_1\leq \cdots\leq \lambda_d$, respectively.  Then there exists a positive constant $\epsilon_{0}(D,K,d,\delta)$, such that if $\lambda_{d}\leq K+\epsilon_0$, then for any $x_0\in V$ there exists $u\in \Lambda(d)$ such that
   $$\Vert\mathrm{dist}_{x_0}-(d/2)-u\Vert_2\leq C(\lambda_d-K)^{\frac{1}{2}},$$
   where $C=C(d,K)$ is a constant depending only on $d$ and $K$.
\end{thm}

Notice that we cannot replace the assumption $\lambda_d\leq K+\epsilon_0$ in Theorem \ref{MainTheoremDistance} by $\lambda_\ell\leq K+\epsilon_0$ for some $\ell<d$. Indeed, there exist graphs $G=(V,w,m)\in \mathcal{G}(D,d,\delta)$ satisfying $CD(K,\infty)$, such that $\lambda_1=\cdots=\lambda_{d-1}=K$ and none of the eigenfunctions to  $\lambda_1=\cdots=\lambda_{d-1}$ can be expressed as a function composed with the distance function. See Example \ref{exampleld} below.

\subsection{Organization of the paper}
    The paper is organized as follows:
In Section \ref{prelim},  we collect basics on graph Laplacian, Bakry-\'Emery curvature, Frobenius distance, etc. 
 In Section \ref{Geostab}, we analyze the stability of eigenfunctions under perturbation of the first $d$ non-zero eigenvalues of a graph $G\in \mathcal{G}(D,d,\delta)$ satisfying $CD(K,\infty)$. In Section \ref{QOT}, we prove Theorem \ref{MainTheorem} and Theorem \ref{MainTheoremDistance}.

\section{Preliminaries}\label{prelim}
Let $G=(V,m,w)$ be a locally finite weighted graph. For two vertices $x,y$, we write $x\sim y$ if $w(x,y)>0$. We define the \textit{combinatorial distance} between two vertices $x$ and $y$ by $\mathrm{dist}(x, y) := \min\{n \ \vert \ \text{there exist } x = x_0\sim \cdots\sim x_n = y\}$. The \textit{diameter} is defined as $\mathrm{diam}(G):=\sup_{x,y\in V}d(x,y)$. We denote by $S_k(x) := \{y \in V \ \vert \ \mathrm{dist}(x, y) = k\}$ and $B_k(x) := \{y \in V \ \vert \ \mathrm{dist}(x, y) \leq k\}$ the sphere and ball of radius $k$ round $x \in V$, respectively.

\subsection{Bakry-\'Emery curvature}
    We define the function space $C(V):=\mathbb{R}^V$. The $graph~Laplacian$   $\Delta:C(V)\to C(V)$ is defined by 
    \begin{equation*}
        \begin{aligned}
            \Delta f(x):=\frac{1}{m(x)}\sum_y w(x,y)(f(y)-f(x)).
        \end{aligned}
    \end{equation*}

Next, we introduce Bakry-\'Emery $\Gamma$-calculus.
\begin{defn}
    For any two functions $f, g\in C(V)$, we define 
\begin{equation*}
    \begin{aligned}
        2\Gamma(f, g):&= \Delta(fg) - f\Delta g - g\Delta f,\\
        2\Gamma_2(f, g) &:= \Delta\Gamma(f, g) - \Gamma(f, \Delta g) - \Gamma(g, \Delta f).
    \end{aligned}
\end{equation*}
\end{defn}
We write $\Gamma(f)=\Gamma(f,f)$ and $\Gamma_2(f)=\Gamma_2(f,f)$, for simplification. Note that $\Gamma$ and $\Gamma_2$ are bilinear and it is direct to check
\[\Gamma(f)(x)=\frac{1}{m(x)}\sum_{y}w(x,y)(f(y)-f(x))^2.\]

\begin{defn}
    Let $G=(V,w,m)$ be a locally finite weighted graph. We say the graph $G$ satisfies $CD(K, N)$ for some \( K \in \mathbb{R} \) and $N \in (0, \infty]$ if,  for every functions $f\in C(V)$ and every vertex $x\in V$,
$$\Gamma_2(f)(x) \geq \frac{1}{N}(\Delta f)^2(x) + K\Gamma f(x).$$
\end{defn}
Here we use the notation $1/\infty=0$. We only use the condition $CD(K,\infty)$ in this paper. A discrete version of the Bonnet–Myers theorem is proved by \cite{diameterbounds}, improving previous results in  \cite{volumedoubling,curvatureandtransportinequalities}.

\begin{thm}[{\cite{diameterbounds}}]\label{thm:BM}
    Let $G = (V,w,m)$ be a locally finite connected weighted graph satisfying $CD(K,\infty)$ for some $K>0$ and $\mathrm{Deg}_{\max}<\infty$. Then the graph $G$ is finite. Moreover, we have 
    \begin{equation}
        \mathrm{diam}(G)\leq \frac{2\mathrm{Deg}_{\max}}{K}.
    \end{equation}
\end{thm}

\subsection{Frobenius distance between weighted graphs}
     In this paper, we employ the Frobenius distance to measure similarities between weighted graphs. This distance is widely used, see, e.g., \cite{graphsimilarity}. Recall that the \textit{Frobenius norm} for a matrix $M= ( M_{ij})$ is given by  $\|M\|_F:=\sqrt{\sum_{i,j}M_{ij}^2}$.

     For a finite weighted graph $G=(V,m,w)$, we consider the following square matrix $A_G=(A_{xy})$, where $A_{x,y}=q(x,y)$ if $x\sim y$ and $A(x,y)=0$ otherwise. Notice that $A_G$ is not necessarily symmetric. 
    
\begin{defn}[Frobenius distance]\label{Frodist}
The \textit{Frobenius distance} between two $n$-vertex weighted graphs $G$ and $H$ is defined as

$$\operatorname{dist}_F(G,H):=\:\min_\pi\|A_G^\pi-A_H\|_F,$$
where the minimization is taken over all all permutations $\pi$ of the vertex set of $G$, and the matrix $A_G^\pi$ is generated by permuting rows and columns of $A_G$ according to $\pi$.
\end{defn}

    The distance can also measure the similarity between two finite weighted graphs of different sizes by adding isolated vertices to the graph with fewer vertices.

\subsection{Equivalence of $\ell_p$-norms}
\begin{defn}  
Let $G=(V,w,m)$ be a locally finite weighted graph. Let $\emptyset\neq W\subset V$,  we define $\ell_{p}(W):=\{f\in C(W):\|f\|_{(W,p)}<\infty\}$ with 
\begin{equation}
\|f\|_{(W,p)}:=\left\{
\begin{aligned}
&\left(\sum_{x\in W}\vert f(x)\vert^p m(x)\right)^{\frac{1}{p}}\quad &\text{for } p\in [1,\infty);\\
&\sup_{x\in W}|f(x)|\quad &\text{for }p=\infty.
\end{aligned}
\right.
\end{equation}
We simply write $\|f\|_{p}$ if $W=V$. 
The space $\ell_2(V)$ is a Hilbert space with the inner product $\langle f,g\rangle:=\sum_{x\in V}f(x)g(x)m(x)$.
\end{defn}

\begin{lem}(Equivalence of norms)\label{equivalence of norm}
   Let $G=(V,w,m)\in \mathcal{G}(D,d,\delta)$ be a graph satisfying $CD(K,\infty)$ for some $K>0$. Then all norms $\Vert \cdot\Vert_p$, $p\in [1,\infty]$ of $C(V)$ are equivalent with respect to a constant $C(D,K,d,\delta,p)$.
    \end{lem}
    \begin{proof}
       
     Let $f$ be any function in $C(V)$.  For $p\in [1,\infty)$, we have $$\frac{1}{\delta}\sum_{x\in V}\vert f(x)\vert^p\leq \sum_{x\in V}\vert f(x)\vert^p m(x)\leq \delta\sum_{x\in V}\vert f(x)\vert^p;$$  For $p=\infty$, the value of $\Vert f\Vert_{\infty}$ is irrelevant to $\delta$. By Theorem \ref{thm:BM}, we have $\mathrm{diam}(G)\leq \frac{2D}{K}$. Thus $\dim(C(V))=\vert V\vert\leq C(D,K,d).$
     
       By the equivalence of norms in finite dimensional spaces, we prove Lemma \ref{equivalence of norm}.    \end{proof}

Let $G=(V,w,m)\in \mathcal{G}(D,d,\delta)$ be a graph.Throughout this work, for convenience, we assume there exists a vertex $x\in V(G)$ such that $m(x)=1$.  It does not make any restriction to the graph $G$, since, in calculation of spectrum and curvature, $m(\cdot)$ appear as denominators. In Definition \ref{definition}, if we keep other conditions unchanged, condition $$\frac{1}{\delta}\leq m(x)\leq \delta, \forall x\in V(G) ~\text{with}~m(x)=1~\text{for some vertex}~x$$ can be replaced by  $$\frac{m(x)}{m(y)}=\frac{q(x,y)}{q(y,x)}\leq \delta,~\forall x\sim y.$$

\section{Geometric stability of eigenfunctions}\label{Geostab}
In Subsection \ref{pertur}, we study the stability of geometric properties of eigenfunctions. With this, in Subsection \ref{resmap}, we are able to approximate distance functions in $\Lambda_0(d)$. In Subsection \ref{verdeg}, we give estimations on vertex measure and degree weight through the properties of the distance functions.

Note that the value of the constants $C(\cdot)$ may change line by line during the proof.

\subsection{Perturbations of eigenvalues }\label{pertur}
 Let $G=(V,w,m)\in \mathcal{G}(D,d,\delta)$ be a graph satisfying $CD(K,\infty)$ for some $K>0$. Let $\Lambda_0(d)$ be the space spanned by a family of  orthonormal eigenfunctions $\varphi_0,\varphi_1,\ldots, \varphi_d$ corresponding to $0=\lambda_0<\lambda_1\leq \cdots\leq \lambda_d$ respectively. In this subsection, we study the stability of functions in  $\Lambda_0(d)$ and  give a quantitative estimate of  part of  \cite[Thoerem 3.4]{rigidityproperties}'s results under small perturbation of the first $d$ non-zero eigenvalues.

 First, we prove the following two lemmas which estimate the gradient.

\begin{lem}\label{Gamma estimate}
   Let $G=(V,w,m)\in \mathcal{G}(D,d,\delta)$ be a graph and $f\in C(V)$. There exists a constant $C(D,d,\delta)$ such that
    \begin{equation}
        \Vert\Gamma f\Vert_1\leq C(D,d,\delta)\Vert f\Vert_2^2. 
    \end{equation}
    \begin{proof}
    For any $f\in C(V)$, we have
        \begin{equation*}
            \begin{aligned}
                 \Vert\Gamma f\Vert_1=&\frac{1}{2}\sum_{x\in V} \sum_{y\sim x}\frac{w(x,y)}{m(x)}(f(y)-f(x))^2 m(x)\\
                 \leq&D\delta\sum_{x\in V} \sum_{y\sim x}(f^2(y)+f^2(x))\\
                 \leq&2Dd\delta\sum_{x\in V}f^2(x)\\
                 \leq&2Dd\delta^2\sum_{x\in V}f^2(x)m(x).
            \end{aligned}
        \end{equation*}
        We prove  Lemma \ref{Gamma estimate}.
    \end{proof}
\end{lem}

\begin{lem}(Gradient Estimation)\label{gradient estimation}
Let $G=(V,w,m)$ be a weighted graph. Let $u,f\in C(V)$ be non-negative functions satisfying $\Delta u\geq -f$, then
\begin{equation}
    \Braket{\Gamma u,1}\leq \Vert u\Vert_1 \Vert f\Vert_\infty.
\end{equation}
\end{lem}
\begin{proof}
By H\"older's inequality, we have
    \begin{equation*}
    \Braket{\Gamma u,1}=-\Braket{u,\Delta u}\leq \Braket{u,f}\leq \Vert u\Vert_1 \Vert f\Vert_\infty.
\end{equation*}
We prove Lemma \ref{gradient estimation}.
\end{proof}

Now, we are ready to study the stability of eigenfunctions. 
\begin{thm}\label{Basicprop}
  Let $G=(V,w,m)\in \mathcal{G}(D,d,\delta)$ be a graph satisfying $CD(K,\infty)$ for some $K>0$. Let $\Lambda_0(d)$ be the space spanned by a family of  orthonormal eigenfunctions $\varphi_0,\varphi_1,\ldots, \varphi_d$ corresponding to $0=\lambda_0<\lambda_1\leq \cdots\leq \lambda_d$, respectively.  If $\lambda_d\leq K+\epsilon$ for some $\epsilon\in(0,1)$, we have for any $\varphi\in \Lambda_0(d)$,
\begin{enumerate}
    \item \begin{equation}
        \label{CD property}
        \Gamma_{2}\varphi\leq K\Gamma \varphi+(K+1)\Vert \varphi\Vert_2^2\delta\epsilon.
        \end{equation}
\end{enumerate}
Furthermore, if  $\mathop{\min}\limits_{w(x,y)>0} w(x,y)\geq\eta>0$, then  the following hold true:
\begin{enumerate}[(a)]
\item There exists a constant depending on $\varphi$ such that
\begin{equation}
\label{graph gradient estimation} 
\left|\Gamma\varphi-const.\right|\leq C(D,K,d,\delta,\eta)\Vert \varphi\Vert_2^2\sqrt\epsilon.
\end{equation}
\item For all $x,z$ with $d(x,z)=2$,~we have 
\begin{equation}
    \label{extension}
\left|{\varphi(z)+\varphi(x)}-\frac{2\sum_{y:x\sim y\sim z}\varphi(y)w(x,y)w(y,z)/m(y)}{\sum_{y:x\sim y\sim z}w(x,y)w(y,z)/m(y)}\right|\leq 2\Vert \varphi\Vert_2(K+1)^{\frac{1}{2}}\delta^{\frac{3}{2}}\eta^{-1} \sqrt{\epsilon}. 
\end{equation}
\end{enumerate}
\end{thm}
    \begin{proof}
        First, we prove \eqref{CD property}.  Let~$0=\lambda_0<\lambda_1=K+\epsilon_1\leq\cdots\leq\lambda_d=K+\epsilon_d\leq K+\epsilon$ be the eigenvalues of the graph Laplacian $-\Delta$ and $\varphi_i$ be eigenfunction corresponding to the $i$-th eigenvalue $\lambda_i$ with  $\Braket{\varphi_i,\varphi_j}=\delta_{ij}$. We spectrally decompose $\varphi=\sum_{i=0}^d a_i\varphi_i\in \Lambda_0(d)$. Then
       \begin{equation*}
        \begin{aligned}
        \Braket{\Gamma_{2}\varphi-K\Gamma \varphi,1}=&\Braket{\Delta \varphi,\Delta \varphi}+K\Braket{\varphi,\Delta \varphi}\\
        =&\sum_{i=1}^d((K+\epsilon_i)^2 a_i^2-K(K+\epsilon_i)a_i^2)\\
        =&\sum(\epsilon_i^2+K\epsilon_i)a_i^2<(K+1)\epsilon\sum a_i^2\\
        \leq& (K+1)\epsilon\Vert \varphi\Vert_2^2.
         \end{aligned}
          \end{equation*}
       Since, for any $x\in V$, $(\Gamma_{2}\varphi-K\Gamma \varphi)(x)m(x)\leq \Braket{\Gamma_{2}\varphi-K\Gamma \varphi,1}$ and~$(\Gamma_{2}\varphi-K\Gamma \varphi)(x)\geq 0$, we prove \eqref{CD property}.
       
       Now, we prove \eqref{graph gradient estimation}. Since $G$ satisfies $CD(K,\infty)$, we have $K\Gamma \varphi\leq \Gamma_2 \varphi$.  We rewrite it as
       \begin{equation*}
           \begin{aligned}
               K\Gamma(\varphi,\sum_{i=1}^d a_i\varphi)\leq& \frac{1}{2}\Delta\Gamma(\varphi)-\Gamma(\varphi,\sum_{i=1}^d a_i\Delta\varphi_i)\\
               =&\frac{1}{2}\Delta\Gamma(\varphi)+\sum(K+\epsilon_i)\Gamma(\varphi, a_i\varphi_i).
           \end{aligned}
       \end{equation*} 
       Then,
       \begin{equation*}
           \begin{aligned}
           \frac{1}{2}\Delta\Gamma(\varphi)\geq& K\Gamma(\varphi,\sum a_i\varphi)-\sum(K+\epsilon_i)\Gamma(\varphi, a_i\varphi)\\
           \geq& -\sum\epsilon_i a_i \Gamma(\varphi,\varphi_i)\\
           \geq& -\epsilon\sum \vert a_i\vert\sqrt{\Gamma(\varphi,\varphi)}\sqrt{\Gamma(\varphi_i,\varphi_i)}\\
           \geq& -\epsilon\sum \vert a_i\vert\sqrt{\Vert\Gamma(\varphi,\varphi)\Vert_{\infty}}\sqrt{\Vert\Gamma(\varphi_i,\varphi_i)\Vert_{\infty}}.
           \end{aligned}
           \end{equation*}
            Applying Lemma \ref{Gamma estimate}, the Cauchy-Schwarz inequality, and the equivalence of norms,
           \begin{equation*}
               \begin{aligned}
                   \frac{1}{2}\Delta\Gamma(\varphi)\geq& -C(D,K,d,\delta)\Vert \varphi\Vert_2 \epsilon\sum\vert a_i\vert\Vert\varphi_i\Vert_2\\
                   \geq& -C(D,K,d,\delta)\Vert \varphi\Vert_2\epsilon\sqrt{\sum{a_i^2}}\\
                   =& -C(D,K,d,\delta)\Vert \varphi\Vert_2^2\epsilon.
               \end{aligned}
           \end{equation*}
           Applying Lemma  \ref{Gamma estimate} and Lemma \ref{gradient estimation}, we obtain
           \begin{equation*}
               \begin{aligned}
                   \Braket{\Gamma\Gamma\varphi,1}\leq C(D,K,d,\delta)\Vert \Gamma \varphi\Vert_1\Vert \varphi\Vert_2^2\epsilon\leq C(D,K,d,\delta)\Vert \varphi\Vert_2^4\epsilon.
               \end{aligned}
           \end{equation*}
           Thus, for any $x\in V $and~$y\sim x$,
            \begin{equation*}
               \begin{aligned}
               w(x,y)(\Gamma\varphi(x)-\Gamma\varphi(y))^2\leq&C(D,K,d,\delta)\Vert \varphi\Vert_2^4\epsilon\\
        \vert\Gamma\varphi(x)-\Gamma\varphi(y)\vert\leq&C(D,K,d,\delta,\eta)\Vert \varphi\Vert_2^2\sqrt{\epsilon}.
               \end{aligned}
               \end{equation*}
               Due to connectedness, we prove \eqref{graph gradient estimation}.

               Next, we prove \eqref{extension}.
               The vector $\varphi_x\in\mathbb{R}^{\#B_2(x)}$  is defined  by the restriction of the function $\varphi$ on 
$B_2(x)$. Let $\Gamma_2(x)$ be the $(\#B_2(x))\times(\#B_2(x))$  symmetric matrix such that  $\Gamma_2\varphi(x)=
\varphi_{x}^{\top}\Gamma_{2}(x)\varphi_{x}$. In fact, for a vertex $z\in S_{2}(x)$, the column $\Gamma_{2}(x)1_{\{z\}}$ of $\Gamma_{2}(x)$ is given as follows (see \cite[Section 2.3]{becurvaturefunctionsongraphs} and \cite[Section 10]{becurvaturefunctionsongraphs}): 

$$(\Gamma_{2}(x))_{x,z}=(\Gamma_{2}(x))_{z,z}=\frac{1}{4m(x)}\sum_{y:x\sim y\sim z}w(x,y)w(y,z)/m(y)\geq  \frac{1}{4}\delta^{-2}\eta^2>0;$$
        For any $y\in S_1(x)$, we have $(\Gamma_2(x))_{y,z}=-\frac{1}{2m(x)}w(x,y)w(y,z)/m(y)$ if~$ y\sim z$ and 0 otherwise. Finally, $(\Gamma_2(x))_{z^{\prime},z}=0$ for any  $z^{\prime}\in S_2(x)$ different from z. Therefore, we have for any $r\in\mathbb{R}$ 
               \begin{equation*}
                   \begin{aligned}
                      \Gamma_{2}\left(\varphi+r1_{\{z\}}\right)(x)=\varphi_{x}^{\top}\Gamma_{2}(x)\varphi_{x}+2r\varphi_{x}^{\top}\Gamma_{2}(x)1_{\{z\}}+r^{2}(\Gamma_{2}(x))_{z,z}. 
                   \end{aligned}
               \end{equation*}
   \eqref{CD property} implies
  \begin{equation*}
      \begin{aligned}
          \Gamma_{2}\left(\varphi+r1_{\{z\}}\right)(x)\leq&K\varphi_{x}^{\top}\Gamma(x)\varphi_{x}+(K+1)\Vert \varphi\Vert_2^2\delta\epsilon+2r\varphi_{x}^{\top}\Gamma_{2}(x)1_{\{z\}}+r^{2}(\Gamma_{2}(x))_{z,z}. 
      \end{aligned}
  \end{equation*}
    Since $\varphi$ and $\varphi+r1_{\{z\}}$ agree on $B_1(x)$, we derive $\Gamma(\varphi)(x)=\Gamma(\varphi+r1_{\{z\}})(x)$. We have
\begin{equation*}
    \begin{aligned}
      (K+1)  \Vert \varphi\Vert_2^2\delta\epsilon+2r\varphi_{x}^{\top}\Gamma_{2}(x)1_{\{z\}}+r^{2}(\Gamma_{2}(x))_{z,z}\geq \Gamma_{2}\left(\varphi+r1_{\{z\}}\right)(x)-K\Gamma(\varphi+r1_{\{z\}})(x)\geq 0
    \end{aligned}
\end{equation*}
stands for any $r\in\mathbb{R}$. By the discriminant formula of quadratic polynomials,
\begin{equation*}
    (2\varphi_{x}^{\top}\Gamma_{2}(x)1_{\{z\}})^2-4(K+1)\Vert \varphi\Vert_2^2\delta\epsilon(\Gamma_{2}(x))_{z,z}\leq0,
\end{equation*}
which yields

\begin{equation}
\label{GammaZZ}
    \begin{aligned}
    \vert\varphi_{x}^{\top}\Gamma_{2}(x)1_{\{z\}}\vert\leq&\Vert \varphi\Vert_2(K+1)^{\frac{1}{2}}\delta^{\frac{1}{2}}\sqrt{\epsilon}((\Gamma_{2}(x))_{z,z})^{\frac{1}{2}}\\
    \leq& \Vert \varphi\Vert_2 (K+1)^{\frac{1}{2}}\delta^{\frac{1}{2}}\sqrt{\epsilon}(\frac{\Gamma_{2}(x))_{z,z}}{\frac{1}{4}\delta^{-2}\eta^2})^{\frac{1}{2}}\frac{1}{2}\delta^{-1}\eta\\
    \leq& \Vert \varphi\Vert_2(K+1)^{\frac{1}{2}} \delta^{\frac{1}{2}}\sqrt{\epsilon}(\frac{\Gamma_{2}(x))_{z,z}}{\frac{1}{4}\delta^{-2}\eta^2})\frac{1}{2}\delta^{-1}\eta\\
    \leq& 2\Vert \varphi\Vert_2(K+1)^{\frac{1}{2}}\delta^{\frac{3}{2}}\eta^{-1}\sqrt{\epsilon}(\Gamma_{2}(x))_{z,z}.
     \end{aligned}
\end{equation}
Note that 
\begin{equation*}
    \begin{aligned}
        \varphi_{x}^{\top}\Gamma_{2}(x)1_{\{z\}} & =\frac{1}{4m(x)}(\varphi(x)+\varphi(z))\sum_{y:x\sim y\sim z}w(x,y)w(y,z)/m(y) \\
 & -\frac{1}{2m(x)}\sum_{y:x\sim y\sim z}\varphi(y)w(x,y)w(y,z)/m(y).
    \end{aligned}
\end{equation*}
Dividing each side of \eqref{GammaZZ} by $(\Gamma_{2}(x))_{z,z}=\frac{1}{4m(x)}\sum_{y:x\sim y\sim z}w(x,y)w(y,z)/m(y)$, \eqref{extension} yields.
\end{proof}

\subsection{Restriction maps}\label{resmap}
The aim of this subsection is to find a function in $\Lambda_0(d)$ which is close to the distance function $f_0:=\mathrm{dist}(x_0,\cdot)$ for some fixed vertex $x_0\in V(G)$.  The strategy is to find a function $\tilde f_0\in \Lambda_0(d)$  equal to $f_0$ on $B_1(x_0)$ and exploit the previous estimates of functions in $\Lambda_0(d)$ in Theorem \ref{Basicprop} to show that such $\tilde f_0$ is exactly the function we want.  

\begin{defn}
\label{defgamma}
    Let $G=(V,w,m)\in \mathcal{G}(D,d,\delta)$ be a graph satisfying $CD(K,\infty)$ for some $K>0$. Let $\Lambda_0(d)$ be the space spanned by a family of  orthonormal eigenfunctions $\varphi_0,\varphi_1,\ldots, \varphi_d$ corresponding to $0=\lambda_0<\lambda_1\leq \cdots\leq \lambda_d$, respectively.  A restriction map $L_x$ with respect to $x\in V$ is defined as:
    \begin{equation}
        \begin{aligned}
             L_x: \Lambda_0(d)&\longrightarrow C(B_1(x))\\
                       \varphi&\longmapsto \varphi\vert_{B_1(x)}.
        \end{aligned}
    \end{equation}
\end{defn}
The following theorem states that the restriction map $L_x$ is bijective and its inverse has bounded norm when the spectral gap deficit is sufficiently small.
\begin{thm} \label{bounded norm}
     Let $G=(V,w,m)\in \mathcal{G}(D,d,\delta)$ be a graph satisfying $CD(K,\infty)$ for some $K>0$  and $\mathop{\min}\limits_{w(x,y)>0} w(x,y)\geq\eta>0$. For any $x\in V(G)$, let $L_x$ be the restriction map from Definition \ref{defgamma}. Then there exists $\epsilon_1(D,K,d,\delta,\eta)>0$, such that for any $0<\epsilon<\epsilon_1$, if  $\lambda_d\leq K+\epsilon$, $L_x$ is bijective. i.e. $deg(x)=d$. Moreover, $\Vert L_x^{-1}\Vert_2\leq C(D,K,d,\delta,\eta)$.
\end{thm}
\begin{proof}
    In fact, we will show that there exists a constant $C(D,K,d,\delta,\eta)>0$ such that
      \begin{equation}
      \label{Inverse}
          \Vert L_x(\varphi)\Vert_2=\Vert\varphi\Vert_{(B_1(x),2)}\geq C \Vert \varphi \Vert_2.
      \end{equation}
      It is direct to see that $\Vert L_x(\varphi)\Vert_2\leq1$. Since $dim(\Lambda_0)=d+1\geq \#B_1(x)$ and $L_x$ is injective, we prove the theorem.
      Now, we prove \eqref{Inverse}.
      
      For any $\varphi\in \Lambda_0(d)$, we assume $\Vert L_x(\varphi)\Vert_{\infty}=\Vert\varphi\Vert_{(B_1(x),\infty)}=:\delta_0$.
      For $z\in B_2(x)$, due to \eqref{extension}, we have
      \begin{equation*}
          \begin{aligned}
              \left| \varphi(z)\right| \leq& \left| {\varphi(z)+\varphi(x)}-\frac{2\sum_{y:x\sim y\sim z}\varphi(y)w(x,y)w(y,z)/m(y)}{\sum_{y:x\sim y\sim z}w(x,y)w(y,z)/m(y)}\right|\\
              &+\left| \varphi(x)-\frac{2\sum_{y:x\sim y\sim z}\varphi(y)w(x,y)w(y,z)/m(y)}{{\sum_{y:x\sim y\sim z}}w(x,y)w(y,z)/m(y)}\right|\\
              \leq&2\Vert \varphi\Vert_2(K+1)^{\frac{1}{2}}\delta^{\frac{3}{2}}\eta^{-1} \sqrt{\epsilon}+3\delta_0.
          \end{aligned}
      \end{equation*}
     For any integer $k\geq 3$, suppose we have $\vert \varphi(w)\vert\leq(3^{k-1}-1)\Vert \varphi\Vert_2(K+1)^{\frac{1}{2}}\delta^{\frac{3}{2}}\eta^{-1} \sqrt{\epsilon}+3^{k-1}\delta_0$, for any $w\in B_{k}(x)$. For any $z\in S_{k+1}(x)$, pick any $x_0\in S_{k-1}(x)$ such that $x_0\sim y\sim z$ for some $y\in S_{k}(x)$.  By \eqref{extension}, we derive
      \begin{equation*}
          \begin{aligned}
              \vert \varphi(z)\vert \leq& \left| {\varphi(z)+\varphi(x_0)}-\frac{2\sum_{y:x_0\sim y\sim z}\varphi(y)w(x_0,y)w(y,z)/m(y)}{\sum_{y:x_0\sim y\sim z}w(x_0,y)w(y,z)/m(y)}\right|\\
              &+\left| \varphi(x_0)-\frac{2\sum_{y:x_0\sim y\sim z}\varphi(y)w(x_0,y)w(y,z)/m(y)}{{\sum_{y:x_0\sim y\sim z}}w(x_0,y)w(y,z)/m(y)}\right|\\
              \leq&2\Vert \varphi\Vert_2(K+1)^{\frac{1}{2}}\delta^{\frac{3}{2}}\eta^{-1} \sqrt{\epsilon}+3((3^{k-1}-1)\Vert \varphi\Vert_2(K+1)^{\frac{1}{2}}\delta^{\frac{3}{2}}\eta^{-1} \sqrt{\epsilon}+3^{k-1}\delta_0)\\
              \leq& (3^{k}-1)\Vert \varphi\Vert_2(K+1)^{\frac{1}{2}}\delta^{\frac{3}{2}}\eta^{-1} \sqrt{\epsilon}+3^{k}\delta_0.
          \end{aligned}
      \end{equation*}
      Since, graph is connected and $\mathrm{diam}(G)\leq \frac{2Deg_{max}}{K}\leq \frac{2D}{K}$. Iterating the above process, we obtain 
      \begin{equation}
          \begin{aligned}
              \Vert \varphi\Vert_{\infty} \leq C(D,K,\delta,\eta)(\Vert \varphi\Vert_2\sqrt{\epsilon}+\delta_0).
          \end{aligned}
      \end{equation}
     By the equivalence of norms, we have
    \begin{equation}
          \begin{aligned}
              \Vert \varphi\Vert_{2} \leq C(D,K,d,\delta,\eta)(\Vert \varphi\Vert_2\sqrt{\epsilon}+\delta_0).
          \end{aligned}
      \end{equation}
      Picking $\sqrt{\epsilon}\leq \frac{1}{2}C^{-1}(D,K,d,\delta,\eta)$, there holds
      \begin{equation}
          \begin{aligned}
              \Vert \varphi\Vert_{2} \leq C(D,K,d,\delta,\eta)\delta_0.
          \end{aligned}
      \end{equation}
      Again, by the the equivalence of norms, we prove \eqref{Inverse}.
      \end{proof}
 
The next Lemma states that it is available to approximate distance functions in $\Lambda_0(d)$ when the spectral gap deficit is sufficiently small. It is crucial for the proof of Theorem \ref{MainTheoremDistance} and build a bridge between eigenfunctions and distance functions. 
\begin{lem}
   \label{uniformdist}
   Let $G=(V,w,m)\in \mathcal{G}(D,d,\delta)$ be a graph satisfying $CD(K,\infty)$ for some $K>0$  and $\mathop{\min}\limits_{w(x,y)>0} w(x,y)\geq\eta>0$.  For any fixed vertex $x_0 \in V(G)$, define the distance function $f_0(\cdot):= \mathrm{dist}(\cdot, x_0) $.
    
    Then, there exist constants $\epsilon_1 = \epsilon_1(D, K, d, \delta, \eta) > 0$ and $C = C(D, K, d, \delta, \eta)$ such that for all $0 < \epsilon < \epsilon_1$, if  $\lambda_d \leq K + \epsilon$,  we have $\tilde{f_0}:= L_{x_0}^{-1}(f_0 \vert_{B_1(x_0)})$ exists and
\begin{equation}
\label{eq:1}
    \Vert f_0-\tilde{f_0}\Vert_{\infty}\leq C(D,K,d,\delta,\eta)\sqrt{\epsilon}.
\end{equation}
\end{lem}
\begin{proof}
    First, let $\epsilon_1$ be the $\epsilon_1(D,d,\delta,K,\eta)$ in Theorem \ref{bounded norm}. Then, $\tilde{f_0}=L_{x_0}^{-1}(f_0 \vert_{B_1(x_0)})$ exists. For all $0 < \epsilon < \epsilon_1$, we prove the Theorem as follows.

Note that, for distance function $f_0(\cdot):= \mathrm{dist}(\cdot, x_0)$, it is direct to check 
$$f_0(z)+f_0(x)=\frac{2\sum_{y:x\sim y\sim z}
            f_0(y)w(x,y)w(y,z)/m(y)}{\sum_{y:x\sim y\sim z}w(x,y)w(y,z)/m(y)},$$
        where  $z\in B_{k+1}(x_0)$ and  $x\in B_{k-1}(x_0)$, $k\in\mathbb{N}^{+}$.
        
Thus, for $z\in B_2(x_0)$, by \eqref{extension}, we have 
    \begin{equation*}
        \begin{aligned}
            \left| f_0(z)-\tilde{f_0}(z)\right|\leq\left| f_0(x)-\tilde{f_0}(x)-\frac{2\sum_{y:x\sim y\sim z}
            (f_0-\tilde{f_0})(y)w(x,y)w(y,z)/m(y)}{\sum_{y:x\sim y\sim z}w(x,y)w(y,z)/m(y)}\right|+2\Vert\tilde{f}_0\Vert_2\delta^{\frac{3}{2}}\eta^{-1} \sqrt{\epsilon}.
        \end{aligned}
    \end{equation*}
    Since $ f _0$ and $\tilde{f_0}$ agree on $B_1(x_0)$, we obtain
    \begin{equation*}
        \begin{aligned}
            \vert f_0(z)-\tilde{f_0}(z)\vert\leq 2\Vert \tilde{f_0}\Vert_2(K+1)^{\frac{1}{2}}\delta^{\frac{3}{2}}\eta^{-1} \sqrt{\epsilon}.
        \end{aligned}
    \end{equation*}
    Note that graph is connected and $\mathrm{diam}(G)\leq \frac{2Deg_{max}}{K}\leq \frac{2D}{K}$. By induction,
    \begin{equation*}
        \begin{aligned}
            \Vert f_0-\tilde{f_0}\Vert_{\infty}\leq C(D,K,d,\delta,\eta)\Vert \tilde{f}_0\Vert_2\sqrt{\epsilon}.
        \end{aligned}
    \end{equation*}
   It is straightforward to check that $ \Vert f_0\Vert_{(B_1(x),2)}\leq \sqrt{d\delta}$. Then due to Theorem \ref{bounded norm},  we prove \eqref{eq:1}.
\end{proof}
Here, we are able to give a discrete analogy of Bertrand's result. Let $G=(V,w,m)\in \mathcal{G}(D,d,\delta)$ be a graph satisfying $CD(K,\infty)$ for some $K>0$. If we only assume the first $l<d$ non-zero eigenvalues close to Lichnerowicz bound, recalling examle \ref{exampleld}, it is possible that no eigenfunctions corresponding to $\lambda_1\dots \lambda_l$ can not be expressed as a function composed with the distance function; on the other hand, the following theorem shows, as long as $deg(x_0)\leq l$, for some $ x_0\in V(G)$, distance function $\mathrm{dist}_{x_0}$ can be expressed as a combination of corresponding eigenfunctions of the first $l$ non-zero eigenvalues up to a constant.
\begin{thm}\label{GenMainTheoremDistance}
    Let $G=(V,w,m)\in \mathcal{G}(D,d,\delta)$ be a graph satisfying $CD(K,\infty)$ for some $K>0$  and 
    \begin{equation}
        \label{eq:edgeweightcondition}  \mathop{\min}\limits_{w(x,y)>0} w(x,y)\geq\eta>0. 
    \end{equation}
 Let $0<l\leq d$ be an integer. Let $\Lambda_0(l)$ be the space spanned by a family of  orthonormal eigenfunctions $\varphi_0,\varphi_1,\ldots, \varphi_l$ corresponding to $0=\lambda_0<\lambda_1\leq \cdots\leq \lambda_l$, respectively.  Then for any $x_0\in V$ with $deg(x_0)\leq \ell$ there exists $u\in \Lambda_0(l)$ such that $$\Vert \mathrm{dist}_{x_0}-u\Vert_2\leq C(\lambda_\ell-K)^{\frac{1}{2}},~\text{where}~ C=C(D,K,d,\delta,\eta).$$
\end{thm}
   
   \begin{proof}[Proof of Theorem \ref{GenMainTheoremDistance}]
       
       By the same proof of Theorem \ref{Basicprop}, If $\lambda_l<K+\epsilon$ for some $\epsilon\in(0,1)$,  \eqref{graph gradient estimation}  and \eqref{extension} hold for $\varphi\in \Lambda_0(l)$.  For any $x_0\in V(G)$ with $deg(x_0)\leq l$, we define the restriction map $L_{x_0}^{l}$ as:
\begin{equation*}
        \begin{aligned}
             L_{x_0}^l: \Lambda_0(l)&\longrightarrow C(B_1(x_0))\\
                       \varphi&\longmapsto \varphi\vert_{B_1(x_0)}.
        \end{aligned}
    \end{equation*}
   Note that here $dim(\Lambda_0(l))=l+1\geq \#B_1(x_0)$. Paralleling the proof of Theorem \ref{bounded norm}, it can be shown that $L_{x_0}^{l}$ is bijective and has a bounded inverse when $\epsilon$ is sufficiently small. Finally, Theorem \ref{GenMainTheoremDistance} follows by Lemma \ref{uniformdist}'s argument. Note that, for large $\epsilon$, we simply pick $u\equiv 0$ and estimate $\Vert \mathrm{dist}_{x_0}\Vert_{\infty}\leq {2D}/{K}.$
   
   Thus, we prove Theorem \ref{GenMainTheoremDistance}.

   \end{proof}
 Notice that for a positive integer $\ell<d$, a  graph $G\in \mathcal{G}(D,d,\delta)$ satisfying $CD(K,\infty)$ with $\lambda_1=\cdots=\lambda_\ell=K$ is not necessarily a hypercube, see \cite[Example 3.2]{rigidityproperties}. Therefore, it is natural that the constant $C$ in Theorem \ref{GenMainTheoremDistance} might depend on more parameters. We remark that the condition \eqref{eq:edgeweightcondition} is also necessary. Since, if an edge with arbitrary small weight is added to a graph $G$, its curvature and eigenvalues only perturb a little while its combiatorial distance can vary dramatically. 
\subsection{Estimations of edge degree and vertex measure}\label{verdeg}

      In this subsection, we prove Theorem \ref{UniformProperty} which states that the edge degree and vertex measure is quantitatively close to a constant in terms of spectral gap deficit.  
\begin{thm}
\label{UniformProperty}

      Let $G=(V,w,m)\in \mathcal{G}(D,d,\delta)$ be a graph satisfying $CD(K,\infty)$ for some $K>0$ and $\mathop{\min}\limits_{w(x,y)>0}w(x,y) \geq \eta>0$. For any fixed vertex $x_0 \in V(G)$, define the distance function $f_0(\cdot):= \mathrm{dist}(\cdot, x_0)$.  For simplicity, we write $\mathrm{Deg}_{max}(G)$ as $D(G)$. 
      
    There exist constants $\epsilon_1 = \epsilon_1(D, K, d, \delta, \eta) > 0$ and $C = C(D, K, d, \delta, \eta)$ such that for all $0 < \epsilon < \epsilon_1$, if 
    $\lambda_d \leq K + \epsilon$, then the following hold:

\begin{enumerate}[(i)]
    \item $\left|\Gamma f_0(x) -D(G)/2\right|\leq C(D,K,d,\delta,\eta)\sqrt{\epsilon}$, for any  $x\in V(G)$.\label{eq:2}
    \item $\left|\mathrm{Deg}(x)-D(G)\right|\leq C(D,K,d,\delta,\eta)\sqrt{\epsilon}$, for any  $x\in V(G)$.\label{eq:3}
    \item $\left|q(x,y)-K/2\right|\leq C(D,K,d,\delta,\eta)\sqrt{\epsilon}$, for any $q(x,y)>0$.\label{eq:4}
    \item $\left|m(x)-1\right|\leq C(D,K,d,\delta,\eta)\sqrt{\epsilon}$, for any $x\in V(G)$.\label{SameVertex}
\end{enumerate}
\end{thm}
\begin{proof}
     First, let $\epsilon_1$ be the $\epsilon_1(D,d,\delta,K,\eta)$ in Theorem \ref{bounded norm}. Then, $\tilde{f_0}=L_{x_0}^{-1}(f_0 \vert_{B_1(x_0)})$ exists. We first prove \eqref{eq:2} and \eqref{eq:3} when   $x_0\in V(G)$ such that $Deg(x_0)=D(G)$. 
     
To simplify the calculation, we shall adopt the following notation:
Let $f,g$ be functions defined on a set $E$ with values in $\mathbb{R}$ and $x\in E$. For any $\epsilon>0$, we write $f(x)=g(x)+C[\epsilon]$,  if $\vert f(x)-g(x)\vert\leq C\epsilon$ for some constant $C$.

   \eqref{eq:2}'. By Lemma \ref{uniformdist}, we have
   \begin{equation*}
       \Gamma f_0(x)=\Gamma (f_0-\tilde{f_0}+\tilde{f_0})(x)=\Gamma\tilde{f_0}(x)+C(D,d,\delta,K,\eta)[\sqrt{\epsilon}].
   \end{equation*}
   On the other hand, applying \eqref{graph gradient estimation},  we have 
   \begin{equation*}
       \Gamma f_0=const+C(D,K,d,\delta,\eta)[\sqrt\epsilon].
   \end{equation*}
   Since $D(G)=2\Gamma f_0(x_0)$, we prove \eqref{eq:2} in such $x_0\in V(G)$.

 \eqref{eq:3}'. By \eqref{eq:2},  for any $x\in V(G)$,
   \begin{equation}\label{Dgeq}
       \begin{aligned}
              D(G)+C(D,K,d,\delta,\eta)[\sqrt\epsilon]=&2\Gamma f_0(x)\\
              =&\sum_{\substack{y\sim x\\f_0(y)\neq f_0(x)}}{\frac{w(x,y)}{m(x)}} \\
 =&\mathrm{Deg}(x)-\sum_{f_0(y)=f_0(x)}\frac{w(x,y)}{m(x)}\\
 \leq& \mathrm{Deg}(x).
       \end{aligned}
   \end{equation}
   For convenience, we  define the inner and outer degree w.r.t. $x_0 \in V$ via
$$d_{\pm}^{x_0}(z) := \sum_{\substack{y \sim z \\ d(y, x_0)=d(z, x_0)\pm 1}} \frac{w(y, z)}{m(z)}$$
 and spherical-degree
$$d_{0}^{x_0}(z) := \sum_{\substack{y \sim z \\ d(y, x_0) = d(z, x_0)}} \frac{w(y, z)}{m(z)}.$$
We remark $ \mathrm{Deg}(z)=d_+^{x_0}(z)+d_0^{x_0}(z)+d_-^{x_0}(z)$ and note that $D(G)\geq \mathrm{Deg}(x)$. By \eqref{Dgeq}, we have for any $x\in V(G)$
\begin{equation}
    \begin{aligned}
        \label{Degree}\mathrm{Deg}(x)=&D(G)+C(D,K,d,\delta,\eta)[\sqrt\epsilon]
        \end{aligned}
        \end{equation}
        and
        \begin{equation*}
    \begin{aligned}
        d_0^{x_0}(x)=&\sum_{f_0(y)=f_0(x)}\frac{w(x,y)}{m(x)}=C(D,K,d,\delta,\eta)[\sqrt\epsilon]. 
    \end{aligned}
\end{equation*}
Since $ \mathrm{Deg}(x)=d_+^{x_0}(x)+d_0^{x_0}(x)+d_-^{x_0}(x)$, we obtain 
\begin{equation}
\label{InOUt}
    d_+^{x_0}(x)+d_-^{x_0}(x)=D(G)+C(D,K,d,\delta,\eta)[\sqrt\epsilon].
\end{equation}

Thus we prove \eqref{eq:2} and \eqref{eq:3} when  $x_0\in V(G)$ such that $\mathrm{Deg}(x_0)=D(G)$. Due to \eqref{Degree}, for any $x_0\in V(G)$, the value of $\mathrm{Deg}(x_0)$ only differs from $D(G)$'s by $C_0(D,K,d,\delta,\eta)\sqrt\epsilon$. Repeat the above proofs of \eqref{eq:2}'and \eqref{eq:3}'. Then the above formulas hold for arbitrary $x_0\in V(G)$.

\eqref{eq:4}.  For any $x_0\in V(G)$, by Theorem \ref{bounded norm} and Lemma \ref{uniformdist}, we have  $$f_0=\varphi+C+C(D,K,d,\delta,\eta)[\sqrt{\epsilon}],$$
where $\varphi\in \Lambda(d)$ with $$\Vert\varphi\Vert_2\leq\Vert \tilde f_0\Vert_2\leq C(D,K,d,\delta,\eta)$$ and $C$ is a constant such that $$f_0\vert_{B_1(x_0)}=(\varphi+C)\vert_{B_1(x_0)}.$$
 Therefore,
   \begin{equation}
       \begin{aligned}
           D(G)+C(D,K,d,\delta,\eta)[\sqrt{\epsilon}]=&\mathrm{Deg}(x_0)=\Delta f_0(x_0)=\Delta \varphi(x_0)\\
          =&-K\varphi(x_0)+C(D,K,d,\delta,\eta)[\epsilon].    
       \end{aligned}
   \end{equation}
   Thus, $$C+C(D,K,d,\delta,\eta)[\sqrt{\epsilon}]=(f_0-\varphi)(x_0)=\frac{D(G)}{K}+C(D,K,d,\delta,\eta)[\sqrt{\epsilon}].$$
   Just take C= $\frac{D(G)}{K}$. Therefore, for any $x\in V(G)$,
   \begin{equation}
       \begin{aligned}
           \Delta f_0(x)=\Delta\varphi_0(x)+C(D,K,d,\delta,\eta)[\sqrt{\epsilon}]=&-K(f_0(x)-C)+C(D,K,d,\delta,\eta)[\sqrt{\epsilon}]\\
           =& D(G)-K d(x,x_0)+C(D,K,d,\delta,\eta)[\sqrt{\epsilon}].
       \end{aligned}
   \end{equation}
   We obtain 
   \begin{equation}
       d_+^{x_0}(x)-d_-^{x_0}(x)=D(G)-K d(x,x_0)+C(D,K,d,\delta,\eta)[\sqrt{\epsilon}].
   \end{equation}
   Combined with \eqref{InOUt}, $$d_-^{x_0}(x)=\frac{1}{2}K d(x,x_0)+C(D,d,\delta,K,\eta)[\sqrt{\epsilon}].$$
   For any $x,y \in V(G)$ with $x\sim y$, we have 
   \begin{equation}
       \begin{aligned}
          q(y,x)=\frac{w(x,y)}{m(y)}=&d_-^{x}(y)=\frac{1}{2}K d(x,y)+C(D,K,d,\delta,\eta)[\sqrt{\epsilon}]\\
          =&\frac{K}{2}+C(D,K,d,\delta,\eta)[\sqrt{\epsilon}]. 
       \end{aligned}
   \end{equation}
  \eqref{SameVertex}. Then, we have $$q(x,y)=q(y,x)+C(D,K,d,\delta,\eta)[\sqrt{\epsilon}],$$which yields $$\frac{m(x)}{m(y)}=1+C(D,K,d,\delta,\eta)[\sqrt{\epsilon}].$$  On the other hand, we always assume that there exists some $x_0\in G$ such that $m(x_0)=1$.  Since $G$ is connected and $\mathrm{diam}(G)\leq \frac{2Deg_{max}}{K}\leq \frac{2D}{K}$, we prove \eqref{SameVertex}.
\end{proof}

\section{Quantitative Obata's Theorem}\label{QOT}
First, we prove the almost-rigidity theorem by contradiction.
\begin{thm}[Almost-rigidity theorem]\label{AMT}
   
   Let $G=(V,w,m)\in \mathcal{G}(D,d,\delta)$ be a graph satisfying $CD(K,\infty)$ for some $K>0$ . There exists an $\epsilon_{0}(D,K,d,\delta)$, such that  if $\lambda_{d}(G)\leq K+\epsilon_0$, then
\begin{enumerate}[(i)]
    \item \label{AR-1} $G$ is a hypercube in the combinatorial sense; 
    \item \label{AR-2} there exists an ~$\eta(D,K,d,\delta)>0$, such that $\mathop{\min}\limits_{w(x,y)>0} w(x,y)\geq\eta(D,K,d,\delta)$, for any $w(x,y)\in w(G)$.
\end{enumerate}
\end{thm}
\begin{proof}
    (1)First, we prove \eqref{AR-1}. Suppose not. There exists a sequence of graphs $$\{G_n=(V_n,w_n,m_n)\}_{n=1}^{\infty}$$ such that $\lambda_{d}(G_n)\leq K+\frac{1}{n}$, while $\{G_n\}_{n=1}^{\infty}$ are not hypercubes in the combinatorial sense. Since $\{G_n\}_{n=1}^{\infty}$ are connected and $diam(G_n)\leq \frac{2Deg_{max}(G_n)}{K}\leq \frac{2D}{K}$, there are only finite many combinatorial structures among all $\{G_n\}_{n=1}^{\infty}$. Thus, we may pick an infinite subsequence of $\{G_n\}_{n=1}^{\infty}$( still denoted as $\{G_n\}_{n=1}^{\infty}$) sharing the same combinatorial structure, say $G=( V,E)$. We label vertices and edges in $\{G_n\}_{n=1}^{\infty}$ with $G=( V,E)$.

    By the definition of  $\mathcal{G}$, $w_n$ and $m_n$ are bounded. There exists a subsequence of  $\{G_n\}_{n=1}^{\infty}$ ( still denoted as  $\{G_n\}_{n=1}^{\infty}$), such that $$w_0(x,y):=\lim_{n\rightarrow +\infty}w_n(x,y)$$ and $$m_0(x):=\lim_{n\rightarrow+\infty}m_n(x)$$ exist, for any $x\sim y$ and $x\in V$.

   Thus, we define the weighted graph $G_0:=(V,w_0,m_0)$. By continuity of  spectrum,   $\lambda_{d}(G_0)=K$. By the Rigidity Theorem (see Theorem \ref{rigiditytheorem}) , $G_0$ is a hypercube.  Note that for any $\{x,y\}\in E$,  $w_0(x,y)$ will not vanish, otherwise $deg(x)\leq d-1$.  This contradicts with $G_0$ being a hypercube. Thus G is isomorphic to $G_0$, a hypercube. This contradicts the assumption. Hence, we prove that  such  $\epsilon_0$ exists.

   (2)Now, we prove \eqref{AR-2}. For the above $\epsilon_0$ in \eqref{AR-1}, suppose not. 
   Since, for any $G=(V,w,m)\in \mathcal{G}(D,d,\delta)$ satisfying $CD(K,\infty)$ for some $K>0$  with $\lambda_{d}(G)<K+\epsilon_0$, $G$ is already a hypercube combinatorially, we may pick a subsequence of $\{G_n\}_{n=1}^{\infty}$ such that ${w_n(x',y')}<\frac{1}{n}$, for some $\{x',y'\}\in E$. By the similar argument of  \eqref{AR-1},  $G_0$ fails to be a hypercube, a contradiction. Thus, such $\eta(D,K,d,\delta)$ exists.
\end{proof}
Now, we are ready to prove our main theorem.
\begin{proof}[Proof of Theorem \ref{MainTheorem}]
    We employ a self-improving argument to prove the constant is irrelevant to $\delta,\eta$. By Theorem \ref{AMT}, there exist $\epsilon_{0}(D,K,d,\delta)$ and  $\eta(D,K,d,\delta)$,  if $\lambda_{d}(G)<K+\epsilon_0$, then $G$ is a hypercube combinatorially and $\min_{w(x,y)>0} w(x,y)\geq\eta(D,K,d,\delta)$.

    Applying Theorem \ref{UniformProperty} for the above $\eta(D,K,d,\delta)$, there exists an $$\epsilon_2(D,K,d,\delta):=\mathop{\min}\{\epsilon_1(D,K,d,\delta,\eta(D,K,d,\delta)),\epsilon_{0}(D,K,d,\delta)\},$$such that when $\lambda_d=K+\epsilon\leq K+\epsilon_2$, we obtain 
          $$\left|m(x)-1\right|\leq C(D,K,d,\delta)\sqrt{\epsilon},\forall x\in V(G)$$
     and $$\left|q(x,y)-\frac{K}{2}\right|\leq C(D,K,d,\delta)\sqrt{\epsilon}, \text{whenever}~ q(x,y)>0.$$
    Set $$\epsilon_3(D,K,d,\delta):=\mathop{\min}\{\epsilon_2(D,K,d,\delta),\frac{1}{4}C^{-2}(D,K,d,\delta),\frac{K^2}{16}C^{-2}(D,K,d,\delta)\}.$$We have $$\left|m(x)-1\right|\leq\frac{1}{2},\forall x\in V(G)$$ and $$\left|q(x,y)-\frac{K}{2}\right|\leq\frac{K}{4}, \text{whenever}~ q(x,y)>0.$$  It yields $D(G)\leq \frac{3dK}{4}$. Applying Theorem \ref{UniformProperty} again,  there exists $$\epsilon_4(D,K,d,\delta):=min\{\epsilon_1(D=\frac{3dK}{4},K,d,\delta=2,\eta=\frac{K}{8}),\epsilon_3(D,K,d,\delta)\}.$$ When $\lambda_d(G)=K+\epsilon\leq K+\epsilon_4(D,K,d,\delta)$, we have  \begin{equation}
          \label{lastestmate}\left|m(x)-1\right|\leq C(K,d)\sqrt{\epsilon},\forall x\in V(G)
    \end{equation} and $$\left|q(x,y)-\frac{K}{2}\right|\leq C(K,d)\sqrt{\epsilon}.$$  Since $G$ and $H$ share the same combinatorial structure, there holds
 
    \begin{equation*}
    \begin{aligned}
        \operatorname{dist}_F(G,H)=&\sqrt{\sum_{x\sim y} (q(x,y)-\frac{K}{2})^2}\\
        \leq & C(K,d)\sqrt{\epsilon}.
    \end{aligned}
    \end{equation*}
    We prove the theorem.
\end{proof}
\begin{proof}[Proof of Theorem \ref{MainTheoremDistance}]
    For any fixed $x_0\in V$, set $f_0:=\mathrm{dist}_{x_0}$. By the equivalence of  norms in $\mathcal{G}(D,K,d,\delta)$, we actually show $$\Vert f_0-\tilde{f_0}\Vert_2\leq C(D,K,d,\delta,\eta)\sqrt{\epsilon}$$  in Lemma \ref{uniformdist}. Picking $\epsilon_4(D,d,\delta,K)$ in the proof of Theorem \ref{MainTheorem}, we have $$\Vert f_0-\tilde{f_0}\Vert_2\leq C(K,d)\sqrt{\epsilon},$$ when $\lambda_d=K+\epsilon<K+\epsilon_4(D,d,\delta,K)$.
    Due to the property of Hilbert spaces, there exists a function $u\in \Lambda(d)$ such that$$\Vert f_0-\frac{\Braket{f_0,1}}{\Braket{1,1}}-u\Vert_2\leq C(K,d)\sqrt{\epsilon}.$$
By \eqref{lastestmate}, we have
\begin{equation*}
        \left|\frac{\Braket{f_0,1}}{\Braket{1,1}}-\frac{d}{2}\right|=  \left| \frac{\sum_{x\in V}f_0(x)(m(x)-1)+\sum_{x\in V}f_0(x)}{\sum_{x\in V}(m(x)-1)+\sum_{x\in V}1}-\frac{\sum_{x\in V}f_0(x)}{\sum_{x\in V}1}\right|\leq C(K,d)\sqrt{\epsilon}.
    \end{equation*}
    By the triangle inequality and the equivalence of norm, we obtain the theorem.
\end{proof}

   \subsection{Examples}
   Given two weighted graphs $G_1=(V_1,w_1,m_1\equiv1)$ and $G_2=(V_2,w_2,m_2\equiv1)$, their Cartesian product $G_1\times G_2=(V(G_1)\times V(G_2),w,m)$ is a weighted graph with vertex set $V(G_1) \times V(G_2)$, 
   $$w((u_1,v_1),(u_2,v_2))=
\begin{cases}
w_1(u_1,u_2), & \text{if } v_1=v_2 \text{ and } u_1\sim u_2, \\
w_2(v_1,v_2), & \text{if } u_1=u_2 \text{ and } v_1\sim v_2, \\
0, &otherwise,
\end{cases}
$$
and $m\equiv1$.  Let $\{\alpha_i\}_{i=0}^{|V_1|-1}$ be the sequence of eigenvalues of $G_1$ and $\{\beta_j\}_{j=0}^{|V_2|-1}$ be the sequence of eigenvalues of $G_2$. The eigenvalues of $G_1\times G_2$ are given by $\{\alpha_{i}+\beta_{j},i=0,1\ldots |V_1|-1,j=0,1\ldots|V_2|-1\}$ with corresponding eigenfunctions $\{h_{\alpha_i,\beta_j}\}$ in the form of 
 \begin{equation*}
        \begin{aligned}
            h_{\alpha_i,\beta_j}: V(G_1)\times V(G_2)&\longrightarrow \mathbb{R}\\
                       (x,y)&\longmapsto f_{\alpha_i}(x)g_{\beta_j}(y),
        \end{aligned}
    \end{equation*}
    where $f_{\alpha_i}$ is the corresponding eigenfunction of $\alpha_i $ and $g_{\beta_j}$ is the corresponding eigenfunction of $\beta_j$.
    
    Further, if $G_1$ and $G_2$ satisfy $CD(K_1,\infty)$ and $CD(K_1,\infty)$, respectively. By \cite[Theorem 1.3]{eigenvalueratios}), their Cartesian product $G_1\times G_2$ satisfies $CD(\min\{K_1,K_2\},\infty)$.
    To clarify the examples, we state the following lemma first.
    \begin{lem}
        \label{distlem}
        Let $G_1=(V_1,w_1,m_1\equiv1)$ and $G_2=(V_2,w_2,m_2\equiv1)$ be two connected weighted graphs with $\left|V(G_1)\right|\geq 2$. If 
        \begin{equation*}
        \begin{aligned}
            h: V(G_1)\times V(G_2)&\longrightarrow \mathbb{R}\\
                       (x,y)&\longmapsto g(y), \text{where } g\in C(V_2)  
        \end{aligned}
    \end{equation*}
    can be expressed as a composite function of distance from any fixed vertex $(x_0,y_0)$ in $G_1 \times G_2$, that is, $h(\cdot)=w(\mathrm{dist}_{(x_0,y_0)}(\cdot))$, then $h\equiv w(0)$, a constant.
    \end{lem}
    \begin{proof}
        We set $e_1:=\max_{x\in V(G_1)}{\mathrm{dist}(x,x_0)}\geq 1$ and $e_2:=\max_{y\in V(G_2)}{\mathrm{dist}(y,y_0)}$. For any integer $0\leq k\leq e_1+e_2$, it is direct to find a vertex $(x,y)$ such that $\mathrm{dist}_{G_1}(x_0,x)=\min\{k,e_1\}$ and $\mathrm{dist}_{G_2}(y_0,y)=\max\{k-e_1,0\}$. We have $$w(k)=h(x,y)=g(y)=h(x_0,y)=w(\max\{k-e_1,0\}).$$Since $G_1\times G_2$ is connected, it yields $w(k)\equiv w(0)$.
    \end{proof}
Now, we are ready to give some examples.
\begin{eg}\label{exampleld}
 
    Let $G=(V,w,m)\in \mathcal{G}(D,d,\delta)$ satisfying $CD(K,\infty)$. If we only assume the first $l\leq d-1$ eigenvalues reach the Lichnerowicz sharpness, there exist graphs such that the first eigenfunction cannot be expressed as a function of distance. Examples are illustrated as follows.
     
   The Cartesian product $G:=H_n(2)\times H_l(1)$, where $n,l\geq 1$ be integers. Note that $H_d(K/2)$ satisfies $CD(K,\infty)$. Therefore, the graph $G$ satisfies $CD(2,\infty)$ with $deg_{max}(G)=n+l$ and its first non-zero Laplacian eigenvalue is equal to $2$ with multiplicity $l$, i.e.
\[2=\lambda_1(G)=\cdots=\lambda_{l}(G)<\lambda_{l+1}(G).\] Since the non-zero eigenvalues of $H_n(2)$ are greater than  2, the corresponding eigenfunction $h$ of  $\lambda=2$ in $H_n \times H_l$ could only be of the form $h=g(y)$, where $g(y)$ is an eigenfunction of  non-zero eigenvalue 2 of $H_l(1)$. Suppose $h$ can be expressed as a function of distance. By Lemma \ref{distlem}, $h$ is a constant, a contradiction.
 \end{eg}
\begin{eg}
   \label{examplecube}
   Let $d>1$ be an integer. There exist eigenfunctions of $H_d$ corresponding to $\lambda_1(H_d)=2$ which can not be expressed as a function of distance.
   
We write $H_d=H_{d-l}\times H_{l}$, where $0<l<d$ is an integer. $H_d$ satisfies  $CD(2,\infty)$. Then, by Lemma \ref{distlem}, any eigenfunction of $ H_d$ in the form of  $h=g(y)$, where $g(y)$ is an eigenfunction of eigenvalue 2 in $H_l$, cannot be a function of distance.
\end{eg}

\printbibliography

@article {rigidityproperties,
    AUTHOR = {Liu, Shiping and M\"unch, Florentin and Peyerimhoff, Norbert},
     TITLE = {Rigidity properties of the hypercube via {B}akry-\'Emery
              curvature},
   JOURNAL = {Math. Ann.},
  FJOURNAL = {Mathematische Annalen},
    VOLUME = {388},
      YEAR = {2024},
    NUMBER = {2},
     PAGES = {1225--1259},
      ISSN = {0025-5831,1432-1807},
   MRCLASS = {05C62 (53C23)},
  MRNUMBER = {4700369},
MRREVIEWER = {Mikhail\ G.\ Katz},
       DOI = {10.1007/s00208-022-02537-y},
       URL = {https://doi.org/10.1007/s00208-022-02537-y},
}

@book {gromov,
    AUTHOR = {Gromov, Misha},
     TITLE = {Metric structures for {R}iemannian and non-{R}iemannian
              spaces},
    SERIES = {Progress in Mathematics},
    VOLUME = {152},
      NOTE = {Based on the 1981 French original,
              With appendices by M.\ Katz, P.\ Pansu and S.\ Semmes,
              Translated from the French by Sean Michael Bates},
 PUBLISHER = {Birkh\"auser Boston, Inc., Boston, MA},
      YEAR = {1999},
     PAGES = {xx+585},
      ISBN = {0-8176-3898-9},
   MRCLASS = {53C23 (53-02)},
  MRNUMBER = {1699320},
MRREVIEWER = {Igor\ Belegradek},
}

@incollection {elworthy,
    AUTHOR = {Elworthy, K. David},
     TITLE = {Manifolds and graphs with mostly positive curvatures},
 BOOKTITLE = {Stochastic analysis and applications ({L}isbon, 1989)},
    SERIES = {Progr. Probab.},
    VOLUME = {26},
     PAGES = {96--110},
 PUBLISHER = {Birkh\"auser Boston, Boston, MA},
      YEAR = {1991},
      ISBN = {0-8176-3567-X},
   MRCLASS = {58G25 (39A12)},
  MRNUMBER = {1168070},
MRREVIEWER = {J\'ozef\ Dodziuk},
}

@article {quantitativeobata,
    AUTHOR = {Cavalletti, Fabio and Mondino, Andrea and Semola, Daniele},
     TITLE = {Quantitative {O}bata's theorem},
   JOURNAL = {Anal. PDE},
  FJOURNAL = {Analysis \& PDE},
    VOLUME = {16},
      YEAR = {2023},
    NUMBER = {6},
     PAGES = {1389--1431},
      ISSN = {2157-5045,1948-206X},
   MRCLASS = {58J50 (53C23)},
  MRNUMBER = {4632377},
MRREVIEWER = {Luis\ Guijarro},
       DOI = {10.2140/apde.2023.16.1389},
       URL = {https://doi.org/10.2140/apde.2023.16.1389},
}

@article {diameterbounds,
    AUTHOR = {Liu, Shiping and M\"unch, Florentin and Peyerimhoff, Norbert},
     TITLE = {Bakry-\'Emery curvature and diameter bounds on graphs},
   JOURNAL = {Calc. Var. Partial Differential Equations},
  FJOURNAL = {Calculus of Variations and Partial Differential Equations},
    VOLUME = {57},
      YEAR = {2018},
    NUMBER = {2},
     PAGES = {Paper No. 67, 9},
      ISSN = {0944-2669,1432-0835},
   MRCLASS = {53C21 (05C81 35K05 35R01)},
  MRNUMBER = {3776357},
MRREVIEWER = {Nelia\ Charalambous},
       DOI = {10.1007/s00526-018-1334-x},
       URL = {https://doi.org/10.1007/s00526-018-1334-x},
}

@article {curvatureaspectsofgraphs,
    AUTHOR = {Bauer, F. and Chung, F. and Lin, Y. and Liu, Y.},
     TITLE = {Curvature aspects of graphs},
   JOURNAL = {Proc. Amer. Math. Soc.},
  FJOURNAL = {Proceedings of the American Mathematical Society},
    VOLUME = {145},
      YEAR = {2017},
    NUMBER = {5},
     PAGES = {2033--2042},
      ISSN = {0002-9939,1088-6826},
   MRCLASS = {31C20 (05C22 31C05 53C21)},
  MRNUMBER = {3611318},
MRREVIEWER = {Riikka\ Kangaslampi},
       DOI = {10.1090/proc/13145},
       URL = {https://doi.org/10.1090/proc/13145},
}

@article {curvatureandhigherorder,
    AUTHOR = {Liu, Shiping and M\"unch, Florentin and Peyerimhoff, Norbert},
     TITLE = {Curvature and higher order {B}user inequalities for the graph
              connection {L}aplacian},
   JOURNAL = {SIAM J. Discrete Math.},
  FJOURNAL = {SIAM Journal on Discrete Mathematics},
    VOLUME = {33},
      YEAR = {2019},
    NUMBER = {1},
     PAGES = {257--305},
      ISSN = {0895-4801,1095-7146},
   MRCLASS = {05C50 (05C76 53C23 58J35 90C22)},
  MRNUMBER = {3907920},
       DOI = {10.1137/16M1056353},
       URL = {https://doi.org/10.1137/16M1056353},
}

@article {becurvaturefunctionsongraphs,
    AUTHOR = {Cushing, David and Liu, Shiping and Peyerimhoff, Norbert},
     TITLE = {Bakry-\'Emery curvature functions on graphs},
   JOURNAL = {Canad. J. Math.},
  FJOURNAL = {Canadian Journal of Mathematics. Journal Canadien de
              Math\'ematiques},
    VOLUME = {72},
      YEAR = {2020},
    NUMBER = {1},
     PAGES = {89--143},
      ISSN = {0008-414X,1496-4279},
   MRCLASS = {05C50 (53C21)},
  MRNUMBER = {4045968},
MRREVIEWER = {Nelia\ Charalambous},
       DOI = {10.4153/cjm-2018-015-4},
       URL = {https://doi.org/10.4153/cjm-2018-015-4},
}

@article {curvatureandtransportinequalities,
    AUTHOR = {Fathi, Max and Shu, Yan},
     TITLE = {Curvature and transport inequalities for {M}arkov chains in
              discrete spaces},
   JOURNAL = {Bernoulli},
  FJOURNAL = {Bernoulli. Official Journal of the Bernoulli Society for
              Mathematical Statistics and Probability},
    VOLUME = {24},
      YEAR = {2018},
    NUMBER = {1},
     PAGES = {672--698},
      ISSN = {1350-7265,1573-9759},
   MRCLASS = {60J27 (53C21 60E15 60J45)},
  MRNUMBER = {3706773},
MRREVIEWER = {Feng-Yu\ Wang},
       DOI = {10.3150/16-BEJ892},
       URL = {https://doi.org/10.3150/16-BEJ892},
}

@article {volumedoubling,
    AUTHOR = {Horn, Paul and Lin, Yong and Liu, Shuang and Yau, Shing-Tung},
     TITLE = {Volume doubling, {P}oincar\'e{} inequality and {G}aussian heat
              kernel estimate for non-negatively curved graphs},
   JOURNAL = {J. Reine Angew. Math.},
  FJOURNAL = {Journal f\"ur die Reine und Angewandte Mathematik. [Crelle's
              Journal]},
    VOLUME = {757},
      YEAR = {2019},
     PAGES = {89--130},
      ISSN = {0075-4102,1435-5345},
   MRCLASS = {58J35 (05C22 35A23 35K08 35R01 53C21)},
  MRNUMBER = {4036571},
MRREVIEWER = {Melanie\ Pivarski},
       DOI = {10.1515/crelle-2017-0038},
       URL = {https://doi.org/10.1515/crelle-2017-0038},
}

@article {petersen,
    AUTHOR = {Petersen, Peter},
     TITLE = {On eigenvalue pinching in positive {R}icci curvature},
   JOURNAL = {Invent. Math.},
  FJOURNAL = {Inventiones Mathematicae},
    VOLUME = {138},
      YEAR = {1999},
    NUMBER = {1},
     PAGES = {1--21},
      ISSN = {0020-9910,1432-1297},
   MRCLASS = {53C20 (53C21 58J50)},
  MRNUMBER = {1714334},
MRREVIEWER = {Joseph\ E.\ Borzellino},
       DOI = {10.1007/s002220050339},
       URL = {https://doi.org/10.1007/s002220050339},
}

@article {pincementbertrand,
    AUTHOR = {Bertrand, J\'er\^ome},
     TITLE = {Pincement spectral en courbure de {R}icci positive},
   JOURNAL = {Comment. Math. Helv.},
  FJOURNAL = {Commentarii Mathematici Helvetici. A Journal of the Swiss
              Mathematical Society},
    VOLUME = {82},
      YEAR = {2007},
    NUMBER = {2},
     PAGES = {323--352},
      ISSN = {0010-2571,1420-8946},
   MRCLASS = {53C20 (53C21 53C23 58J50)},
  MRNUMBER = {2319931},
MRREVIEWER = {Alexandre\ Engoulatov},
       DOI = {10.4171/CMH/93},
       URL = {https://doi.org/10.4171/CMH/93},
}

@incollection {graphsimilarity,
    AUTHOR = {Grohe, Martin and Rattan, Gaurav and Woeginger, Gerhard J.},
     TITLE = {Graph similarity and approximate isomorphism},
 BOOKTITLE = {43rd {I}nternational {S}ymposium on {M}athematical
              {F}oundations of {C}omputer {S}cience},
    SERIES = {LIPIcs. Leibniz Int. Proc. Inform.},
    VOLUME = {117},
     PAGES = {Art. No. 20, 16},
 PUBLISHER = {Schloss Dagstuhl. Leibniz-Zent. Inform., Wadern},
      YEAR = {2018},
      ISBN = {978-3-95977-086-6},
   MRCLASS = {68R10 (05C60 68Q17 68Q25 68W40 90C35)},
  MRNUMBER = {3853975},
}

@book {bakry,
    AUTHOR = {Bakry, Dominique and Gentil, Ivan and Ledoux, Michel},
     TITLE = {Analysis and geometry of {M}arkov diffusion operators},
    SERIES = {Grundlehren der mathematischen Wissenschaften},
    VOLUME = {348},
 PUBLISHER = {Springer, Cham},
      YEAR = {2014},
     PAGES = {xx+552},
     %ISBN = {978-3-319-00226-2; 978-3-319-00227-9},
MRCLASS = {60J25 (58J65 60J35 60J60)},
  MRNUMBER = {3155209},
MRREVIEWER = {Ming\ Liao},
       DOI = {10.1007/978-3-319-00227-9},
       URL = {https://doi.org/10.1007/978-3-319-00227-9},
}

@article {riccicurlinyau,
    AUTHOR = {Lin, Yong and Yau, Shing-Tung},
     TITLE = {Ricci curvature and eigenvalue estimate on locally finite
              graphs},
   JOURNAL = {Math. Res. Lett.},
  FJOURNAL = {Mathematical Research Letters},
    VOLUME = {17},
      YEAR = {2010},
    NUMBER = {2},
     PAGES = {343--356},
      ISSN = {1073-2780},
   MRCLASS = {05C10 (05C50)},
  MRNUMBER = {2644381},
       DOI = {10.4310/MRL.2010.v17.n2.a13},
       URL = {https://doi.org/10.4310/MRL.2010.v17.n2.a13},
}

@incollection {michael,
    AUTHOR = {Schmuckenschl\"ager, Michael},
     TITLE = {Curvature of nonlocal {M}arkov generators},
 BOOKTITLE = {Convex geometric analysis ({B}erkeley, {CA}, 1996)},
    SERIES = {Math. Sci. Res. Inst. Publ.},
    VOLUME = {34},
     PAGES = {189--197},
 PUBLISHER = {Cambridge Univ. Press, Cambridge},
      YEAR = {1999},
      ISBN = {0-521-64259-0},
   MRCLASS = {60J75 (52C99 60J35)},
  MRNUMBER = {1665591},
MRREVIEWER = {Joerg-Uwe\ Loebus},
}

@article {eigenvalueratios,
    AUTHOR = {Liu, Shiping and Peyerimhoff, Norbert},
     TITLE = {Eigenvalue ratios of non-negatively curved graphs},
   JOURNAL = {Combin. Probab. Comput.},
  FJOURNAL = {Combinatorics, Probability and Computing},
    VOLUME = {27},
      YEAR = {2018},
    NUMBER = {5},
     PAGES = {829--850},
      ISSN = {0963-5483,1469-2163},
   MRCLASS = {05C50 (53C23)},
  MRNUMBER = {3868013},
MRREVIEWER = {Ivan\ Izmestiev},
       DOI = {10.1017/S0963548318000214},
       URL = {https://doi.org/10.1017/S0963548318000214},
}

@article {aubry,
    AUTHOR = {Aubry, Erwann},
     TITLE = {Pincement sur le spectre et le volume en courbure de {R}icci
              positive},
   JOURNAL = {Ann. Sci. \'Ecole Norm. Sup. (4)},
  FJOURNAL = {Annales Scientifiques de l'\'Ecole Normale Sup\'erieure.
              Quatri\`eme S\'erie},
    VOLUME = {38},
      YEAR = {2005},
    NUMBER = {3},
     PAGES = {387--405},
      ISSN = {0012-9593},
   MRCLASS = {53C21 (53C20)},
  MRNUMBER = {2166339},
MRREVIEWER = {Nader\ Yeganefar},
       DOI = {10.1016/j.ansens.2005.01.002},
       URL = {https://doi.org/10.1016/j.ansens.2005.01.002},
}

@article {kahler,
    AUTHOR = {Chu, Jianchun and Wang, Feng and Zhang, Kewei},
     TITLE = {The rigidity of eigenvalues on {K}\"ahler manifolds with
              positive {R}icci lower bound},
   JOURNAL = {J. Reine Angew. Math.},
  FJOURNAL = {Journal f\"ur die Reine und Angewandte Mathematik. [Crelle's
              Journal]},
    VOLUME = {820},
      YEAR = {2025},
     PAGES = {213--233},
      ISSN = {0075-4102,1435-5345},
   MRCLASS = {32Q15},
  MRNUMBER = {4871397},
MRREVIEWER = {Kai\ Tang},
       DOI = {10.1515/crelle-2024-0101},
       URL = {https://doi.org/10.1515/crelle-2024-0101},
}
\end{document}